\documentclass[11pt,a4paper]{article}

\usepackage{authblk}

\title{On Global Hyperbolicity of Spacetimes: Some Recent Advances and Open Problems}
\title{On Global Hyperbolicity of Spacetimes: Topological and Functional Analytic Aspects}
\title{On Global Hyperbolicity of Spacetimes: Recent Connections Between Topology and Functional Analysis}
\title{On Global Hyperbolicity of Spacetimes: Connecting Topology and Functional Analysis}
\title{On Global Hyperbolicity of Spacetimes: Topology Meets Functional Analysis}

\title{On Global Hyperbolicity of Spacetimes: \\ Topology Meets Functional Analysis}
\author[1]{Felix Finster\footnote{finster@ur.de}}
\author[2]{Albert Much\footnote{much@itp.uni-leipzig.de}} 
\author[3]{Kyriakos Papadopoulos\footnote{kyriakos@sci.kuniv.edu.kw}} 
\affil[1]{Fakult\"at f\"ur Mathematik,
	Universit\"at Regensburg,
	D-93040 Regensburg }
\affil[2]{Institut f\"ur Theoretische Physik\\ Universit\"at Leipzig\\ D-04103 Leipzig} 
\affil[3]{Department of Mathematics, Kuwait University,   Safat 13060, Kuwait}
\date{August 2021}

\usepackage{float}            
\usepackage[T1]{fontenc} 
\usepackage[cmex10]{amsmath}  
\usepackage{amstext}          
\usepackage{mathrsfs}
\usepackage{amsfonts}         
\usepackage{amssymb}          
\usepackage{amsthm}    
\usepackage{setspace}
\usepackage{bm}               
\usepackage{enumerate}        
\usepackage[bottom]{footmisc} 
\usepackage{array}            
\usepackage{textcomp}         
\usepackage{pdfpages}         
\usepackage{parskip}          
\usepackage[right]{eurosym}   
\usepackage{xcolor}           
\usepackage[square,numbers]{natbib}	  
\usepackage{tikz}
\usetikzlibrary{matrix}

\usepackage{amsmath}  
\usepackage[colorlinks,pdfpagelabels,pdfstartview = FitH,bookmarksopen = true,bookmarksnumbered =
true,linkcolor = black,plainpages = false,hypertexnames = false,citecolor = black]{hyperref}
\usepackage[a4paper,textwidth=17cm,textheight=24cm,headheight=14pt]{geometry}
\allowdisplaybreaks

\newtheorem{definition}{Definition}[section]
\newtheorem{theorem}[definition]{Theorem}
\newtheorem{proposition}[definition]{Proposition}

\newcommand{\R}{\mathbb{R}}

\newcommand{\I}{\mathbb{I}}

\usepackage{fancyhdr}
\numberwithin{equation}{section} 

\begin{document}
	\maketitle \abstract{This chapter is an up-to-date account of results on globally hyperbolic spacetimes,  and serves several purposes. We begin with the exposition of results from a foundational level, where the main tools are order theory and general topology, we continue with results of a more geometric nature, and we conclude with results that are related to current research in theoretical physics. In each case, we list a number of open questions and     formulate, for a class of spacetimes, an interesting connection between global  hyperbolicity of a manifold and the   geodesic completeness of its corresponding space-like surfaces. This connection is substantial for the proof of essential self-adjointness of a   class of pseudo differential operators, that stem from relativistic quantum field theory. }
	\tableofcontents

	\section{Introduction} 
	
	While in Riemannian geometry it is natural to consider geodesic completeness, in Lorentzian geometry -- and in particular in spacetime geometry -- it is more sensible to consider global hyperbolicity as a condition that two events, that are chronologically related, can be joined by a maximal timelike geodesic. This is due to the validity of the Hopf-Rinow theorem, whose naive analogue fails to exist in Lorentzian geometry (see \cite{Galloway2}).
	\\\\
	The advantage of global hyperbolicity is its multifaceted  perspectives. For example, it can be examined from an order-theoretic and topological  perspective, as the strongest of the causal conditions on a spacetime (the causal diamonds are compact - see \cite{Penrose-difftopology}) or from a purely geometric  viewpoint (existence of a Cauchy surface - see \cite{Hawking-Ellis}). The equivalence of such seemingly different statements is what supplies mathematical physics with mathematical tools and physical insights. The importance of such an interdisciplinary work cannot be overstated; in mathematical  physics this is significant in order to construct quantum field theories in curved spacetimes in a mathematically rigorous fashion. In particular, the input that comes from a topological  direction is fruitful   towards a more abstract formulation. From a mathematical point of view, it is important to see the applications of general theorems and to examine and push the scope of their validity. 
	\\\\
	It is worth mentioning that there is a theoretical background  behind the notion of globally hyperbolic spacetimes that uses tools from domain theory, and interconnects mathematical physics with the foundations of mathematics. A recent example is  reference \cite{Minguzzi2}, where the authors showed that globally hyperbolic spacetimes belong to a category\footnote{In the frame of the field of mathematics called Category Theory; see, for example, \cite{Category}.}, which is equivalent to integral domains. Integral domains are partially ordered sets which carry, intrinsically, notions of completeness and approximation, and were used in theoretical computer science (see \cite{Scott}). The equivalence of causality between events to an order on regions of a spacetime suggests that questions about spacetime can be translated to questions on domain theory. This seems of  great interest, because the type of domains  called $\omega$-continuous are the ideal completions of countable abstract bases, so that a spacetime can be reconstructed (in a purely order-theoretical manner) from a dense discrete set. In particular, this suggests that a globally hyperbolic spacetime is linked to  a discrete object. 
	\\\\
	Strongly causal spacetimes, and thus globally hyperbolic ones which are strongly causal by definition,  are conformal to spacetimes in which all null geodesics are complete, that is, their null geodesics can be extended to infinite values of their affine parameters, as the following theorem suggests (see \cite{Clarke}). 
	\begin{theorem} 
		If $M$ is a strongly causal spacetime with metric $g$, then there is a $C^\infty$-function $\Omega$, such that null geodesics with respect to $\Omega^2g$ are complete. 
	\end{theorem}

	As for timelike geodesic completeness, even if there is no obvious relation to global hyperbolicity, there has been shown that under particular sectional or Ricci curvature conditions, if a spacelike hypersurface is future-timelike geodesically complete, then it is globally hyperbolic; see \cite{Galloway2}. 
	Next, we turn to the spacelike geodesic completeness of the Cauchy surface of a globally hyperbolic manifold. \\\\  An interesting result relating     spacelike geodesic completeness to global hyperbolicity was given in \cite[Proposition~5.3]{K78}. The author proved that an ultra-static spacetime $(M,g)$ is globally hyperbolic if and only if the global Cauchy surface is geodesically complete.  The physical advantage of working in globally hyperbolic manifolds is that the wave equation (in what follows referred to as the Klein-Gordon equation), i.e.\ the differential equation describing the dynamics of a spin zero particle, is well-posed. In particular, given smooth initial data on the Cauchy hypersurface, there exists a corresponding global smooth solution. For these manifolds, it is in general possible to rewrite the second order differential equation as a first order system of equations; this is referred to as the Hamiltonian formulation.  The generator of the time evolution can then be written as a $2 \times 2$ matrix differential operator. In order to define a corresponding quantum field theory using the   Hamiltonian formulation,  a necessary requirement   is the essential  self-adjointness of weighted Laplace-Beltrami operators stemming from the Klein-Gordon equation. The essential self-adjointness of these operators depends  on the spacelike geodesic completeness of the  Cauchy surfaces.  In what follows, we   expand more on the obtained results.

	\section{Some Preliminaries}

	\subsection{Elements of Domain Theory and General Topology} 
	
	In this Section we   list some definitions, which we need in section 3, where we   discuss   a link between a certain type of partially ordered sets (posets) with globally hyperbolic spacetimes.
	\\\\
	Throughout the text, the {\em power set} of a set $X$ is denoted by $\mathcal{P}(X)$, and it is defined as the set of all subsets of $X$.

	\begin{definition}\label{D1}
		If $(A,\prec)$ is a poset, then:
		\\
		\begin{enumerate}
			\item a nonempty set $D \in \mathcal{P}(A)$ is called {\em directed}, if for every $x,y \in D$, there exists $z \in D$, such that both $x \prec z$ and $y \prec z$.

			\item (dually to (1)) a nonempty set $F \in \mathcal{P}(A)$ is called {\em filtered}, if for every $x,y \in F$, there exists $z \in F$, such that $z \prec x$ and $z \prec y$.

			\item  a nonempty set $L \in \mathcal{P}(A)$ is called {\em lower set}, if for every $x \in A$ and for every $y \in L$,  $x \prec y$ implies that $x \in L$.

			\item $(A,\prec)$ is a {\em dcpo} (directed, complete poset), if every directed set has a supremum, that is, if $D \in \mathcal{P}(A)$ is directed, then $D$ has a least upper bound.

			\item $(A,\prec)$ is {\em continuous}, if there exists $B \in \mathcal{P}(A)$, such that $B \cap \downarrow x$ contains a directed set with supremum, for all $x \in A$, where $\downarrow x = \{a \in A : a \ll x\}$ and where $\ll$ is defined as $x \ll y$, if for all directed sets $D$ with supremum, $y \prec \sup D$ implies that there exists $d \in D$, such that $x \prec d$.
			
			\item $(A,\prec)$ is {\em bicontinuous}, if it is continuous and for every $x,y \in A$, $x \ll y$ if for every filtered set $F \in \mathcal{A}$ with infimum, $\inf F \prec x$ implies that there exists $f \in F$, such that $f \prec y$ and also, for every $x \in A$, $\uparrow  x$ is filtered, with infimum $x$ (and where, dually to $\downarrow x$, $\uparrow x = \{a \in A : x \ll a\}$).
			
			\item if $(A,\prec)$ is bicontinuous, then its {\em interval topology} is the topology that has a basis that consists of the intervals $(a,b):=\{x \in A : a \ll x \ll b\}$.
			
			\item $(A,\prec)$ is {\em globally hyperbolic}, if it is bicontinuous and the closure of the intervals $(a,b)$, that is $\overline{(a,b)} := [a,b] = \{x : a \preceq x \preceq b\}$, is compact in the interval topology on $A$.
			
			\item A {\em domain} is a continuous dcpo. An example of an {\em interval domain} is the domain of compact intervals of the real line.
			
			\item Let $\Uparrow x = \{a \in D: x \ll a\}$, where $D$ is a dcpo. Then the class $\{\Uparrow x : x \in D\}$ forms a base for the {\em Scott topology} on a continuous poset.
			
		\end{enumerate}
	\end{definition} 
	
	\subsection{Causal Relations in Spacetime}
	It is standard (see \cite{Penrose-difftopology}) in a spacetime to consider two partial orders, the {\em chronological} order $\ll$, which is considered to be irreflexive if one wants to avoid closed timelike curves, and
	the {\em causal order} $\prec$ which is reflexive, and are defined, respectively,
	as follows:\\ 
	\begin{enumerate}
		
		\item $x \ll y$ if and only if $y \in C_+^T(x)$ and
		\\
		\item $x \prec y$ if and only if $y \in C_+^T(x) \cup C_+^L(x)$
		
		where $C_+^T(x)$ denotes the future {\em time-cone} of an event $x$, $C_+^L(x)$ its future {\em light-cone} and $C_+^T(L) \cup C_+^L(x)$ its future {\em causal-cone} (for an analytical exposition of time-coordinates, space-coordinates and the spacetime ``metric'', that leads to the definition of these cones see \cite{Hawking-Ellis}; for a detailed exposition of these definitions on Minkowski space see \cite{Waleed}).

	\end{enumerate}$\,$\\
	In addition, the reflexive relation {\em horismos} $\rightarrow$ is defined as $x \rightarrow y \textrm{ iff } x \prec y \textrm{ but not } x \ll y $. 	Naturally, in a spacetime $M$ one can define the sets $I^+(x) = \{y \in M :x \ll y \}$ and $J^+(x) = \{y \in M: x \prec y\}$; $I^-(x)$ and $J^-(x)$ are defined dually. 
	\\\\	
	In Section 3, we will  examine the relation between $\ll$ in Definition \ref{D1} (5) and the chronological order $\ll$ that will be defined in 2.2 (1). Furthermore, the interval topology of Definition 2.1 (7), defined for bicontinuous posets, will be called the {\em Alexandrov topology} $T_A$ for a spacetime $M$, whenever $\ll$ is the chronological order on $M$. 
	
	\subsection{Weighted Riemannian Manifolds and All That}
	In the following we define  a  weighted  manifold, \cite[Chapter 3.6, Definition 3.17]{AG1}.
	
	\begin{definition}
		A triple $(\Sigma,	g_{\Sigma},\mu)$ is called a \emph{weighted manifold}  if $(\Sigma,	g_{\Sigma})$ is a Riemannian manifold and $\mu$ is a measure on $\Sigma$ with a smooth and everywhere positive density function $\rho$, i.e., $d\mu=\rho\,d\Sigma$.   	A \emph{weighted Hilbert space}, denoted by $L^2(\Sigma, \mu)$, is given  as the space of all  square-integrable functions  on the manifold $\Sigma$ with respect to the measure $\mu$.  	The corresponding  \emph{weighted Laplace-Beltrami operator} (also called the Dirichlet-Laplace operator),  denoted by $\Delta_{g_{\Sigma}, \mu}$, is given by  \begin{equation}\label{eq:wlbop}\Delta_{g_{\Sigma}, \mu}=\frac{1}{\rho \sqrt{|g_{\Sigma}|}}\partial_i(\rho\sqrt{|g_{\Sigma}|}\,g_{\Sigma}^{ij}\partial_j).
		\end{equation} 
	\end{definition}
	
	In regards to the upcoming discussion on essential self-adjointness,  we need an essential result that was given in~\cite[Theorem 1]{S01}.  
	\begin{theorem}\label{T5}
		Let the Riemannian manifold $(\Sigma,g_{\Sigma})$ be complete,  $\mu$ be a measure on $\Sigma$ with a smooth and everywhere positive density function $\rho$ and suppose that  the potential  $Y\in L^2_{loc}(\Sigma)$
		point-wise. Furthermore, let $A\in\Lambda^1_{(1)}(\Sigma)$ and let the   operator  $L=-{\Delta}^{A}_{{g_{\Sigma},\mu}}+Y$, where the minimally coupled Laplace Beltrami operator is defined by
		\begin{align*}
			-{\Delta}^{A}_{{g_{\Sigma},\mu}} :=-(\rho\sqrt{|g_{\Sigma}|})^{-1}(\partial-A)^{*}_{i}(   \sqrt{|g_{\Sigma}|}  \, \rho\, g^{ij}_{\Sigma}\,    (\partial-A)_{j} ) 
		\end{align*}	 
		be 	semi-bounded from below. Then, the operator    $L$ is an essentially self-adjoint operator on $C_0^{\infty}(\Sigma)$.
	\end{theorem} 	 Here   we write $f\in L^2_{loc}(\Sigma)$  for 
	a local $L^2(\Sigma)$ function $f$  that is an element of the   Hilbert space $L^2(\Sigma)$ on every compact subset of the manifold $\Sigma$, and we denote $\Lambda^p_{(k)}(\Sigma)$ as  the set of all $k$-smooth
	(i.e.\ of the class $C^k$) complex-valued $p$-forms on $\Sigma$.  Note that the above result is (apart from the potential term and the gauge field) a reformulation of the result of \cite{STR83} for weighted Riemannian manifolds.  
	\section{Globally Hyperbolic Spacetimes}
	
	Classically, when we talk about global hyperbolicity, in terms of the causal structure of a spacetime, we consider the following definition  (see \cite{Penrose-difftopology}).
	
	\begin{definition}
		A spacetime $M$ is {\em globally hyperbolic}, if and only if $M$ is strongly causal and every set $J^+(x) \cap J^-(y)$ is compact, for some $x,y \in M$.
	\end{definition}
	In \cite{B03} it was proven that any globally hyperbolic spacetime admits a  \emph{smooth} foliation into Cauchy surfaces \cite[Theorem 1.1]{B03}. Moreover, the induced metric of such a globally hyperbolic spacetime admits a specific form \cite[Theorem 1.1]{B05}.
	
	\begin{theorem}\label{T1}
		Let $(M,{g})$ be an $(n+1)$-dimensional globally hyperbolic spacetime. Then  it is isometric to
		the smooth product manifold $\mathbb{R}\times \Sigma$ with a metric $g$, i.e.,
		\begin{equation} 
			g=-N^2 dt^2+h_{ij}dx^idx^j,
			\label{m}\end{equation}
		where $\Sigma$ is a smooth $n$-manifold, $t:\mathbb{R}\times  \Sigma\mapsto\mathbb{R}$   is the natural
		projection, $N: \mathbb{R}\times  \Sigma \mapsto (0,\infty)$ a smooth function, and $\mathbf{h}$ a $2$-covariant symmetric tensor field on $\mathbb{R}\times  \Sigma$, satisfying the following condition:  Each hypersurface  $\Sigma_t\subset M$ at constant $t$ is a Cauchy surface, and the restriction $\mathbf{h}(t)$ of $\mathbf{h}$ to such a  $\Sigma_t$ is a Riemannian metric (i.e.\  $\Sigma_t$ is spacelike).
	\end{theorem} 
	
	Hounnonkpe and Minguzzi proved that the (strong-)causality condition is unnecessary in the definition of global hyperbolicity, for a reasonable (meaning non-compact) spacetime of dimension strictly greater than $2$ (see \cite{Minguzzi2}).  
	
	\begin{theorem}[Hounnonkpe-Minguzzi]
		A non-compact spacetime of dimension strictly greater than $2$ is globally hyperbolic, if and only if the causal diamonds are compact in the  topology, $T_A$.
	\end{theorem}
	
	Recently, there has been a generous step towards  an understanding of the  nature of globally hyperbolic spacetimes down to the most fundamental level. In \cite{Domains}, it has been shown that the structure of such spacetimes is equivalent to the structure of interval domains, which are purely order-theoretic objects. Let us have a closer look at the main results.

	\begin{definition} 
		A set $B$ equipped with a transitive relation $\ll$ is an {\em abstract basis}, if $\ll$ is $-$-interpolative, that is for all $S \in \mathcal{P}(B)$, such that $S$ is finite, if $S \ll x$, then there exists  $y\in B$, such that $S \ll y \ll x$, where by $S \ll x$ one means that $y \ll x$, for all $y \in S$.
	\end{definition}
	
	By $int(C)$ one denotes the set $\{(a,b): a \ll b\}$, where 
	$(a,b) \ll (c,d)$ if and only if $a \ll c$ and $d \ll b$. It is easy to see that, if $(B,\ll)$ is an abstract basis which is $-$- and $+$-interpolative (where $+$-interpolation is the dual to $-$-interpolation), then the set $(int(C),\ll)$ is an abstract basis.
	
	\begin{definition}
		An {\em ideal} of an abstract basis $(B,\ll)$ is a nonempty set $I \in \mathcal{P}(B)$, which is lower and directed. 	The set $(\overline{B},\subset)$ of all ideals of the abstract basis $(B,\ll)$ is a poset and is called the {\em ideal completion} of $B$.
	\end{definition}

	Before introducing the main theorems of \cite{Domains}, we remind that since spacetimes are considered naturally as being second countable  and second countability implies separability, a spacetime admits a countable dense subset. In the theorem that follows a countable set, which is equipped with a causal relation, determines the entire space.
	
	\begin{theorem}[Martin-Panangaden]\label{Martin} If $C$ is a countable dense subset of a globally hyperbolic spacetime $M$, where $\ll$ denotes chronology, then $\max IC \simeq M$, where the maximal elements are equipped with the Scott topology.
	\end{theorem}
	
	Theorem \ref{Martin} can be written more generally as follows:
	
	\begin{theorem}\label{Martinrev} If $C$ is a countable dense subset of a globally hyperbolic set $M$, then $\max IC \simeq M$.
	\end{theorem}

	The proof of Theorem \ref{Martinrev} is the same as that of Theorem \ref{Martin}. Since $M$ is bicontinuous by definition, $(C,\ll)$ will be an abstract basis and so $(int(C),\ll)$ will be an abstract basis for $IM$. Since ideal completions for bases of $IM$ are isomorphic to $IM$, the result follows immediately. The result of Theorem \ref{Martin} follows from the fact that the relation $\ll$ of Definition \ref{D1} (5), which reads $x \ll y$ as ``$x$ approximates $y$'', coincides with the chronological partial order $\ll$ in a globally hyperbolic spacetime (see \cite{Causality}, Proposition 4.4).

	\begin{theorem}[Martin-Panangaden]\label{Martinfinal}
		The category of globally hyperbolic posets is equivalent to the category of interval domains.
	\end{theorem} 
	
	The proof of Theorem \ref{Martinfinal} depends on several technical results that precede it in \cite{Domains}, but the result on its own is powerful; within a globally hyperbolic spacetime, one can convert questions of a physical significance, to questions in domain theory (and vice versa).  
	\\
	
	{\bf{ Question 1:}} Consider a globally hyperbolic poset $(A,\prec)$. Does the compactness of the intervals $[a,b]$, under the interval topology,  imply the bicontinuity of $(A,\prec)$, in a similar fashion that the compactness in the closed diamonds in a spacetime $M$, under the Alexandrov topology $T_A$, implies the (strong) causality of $M$ (for dimensions strictly greater than $2$)? This is a reasonable question, given the results in \cite{Minguzzi2} as well as the main result of \cite{Domains}, which guarantee  the equivalence (up to category theory) of interval domains and globally hyperbolic spacetimes. The only problem, towards \cite{Minguzzi2}, of considering a globally hyperbolic poset, is that we are  left only with an order-theoretic structure; what would be, if there is such, the equivalent condition in \cite{Minguzzi2} of ``dimension strictly greater than $2$''? Is there a possibility that the condition on the dimension applies only to globally hyperbolic spacetimes, while the proposed conjecture holds in globally hyperbolic posets without any restriction?

	In the section that follows, we consider a spacetime, and we discuss how topology can affect the way that we look at it.
	
	\section{Different Candidates for a Topology: Which is the Most ``Fruitful'' Choice?}
	
	When one considers a spacetime manifold $(M,g)$, the spacetime topology is traditionally  taken to be the manifold topology $T_M$, a topology which does not incorporate the causal-structure of $M$, a structure which is linked to the Lorentzian metric $g$. This problem was first addressed by Zeeman in his papers \cite{Zeeman2} and \cite{Zeeman1}, where he restricted the discussion to Minkowski space. It was then extended to curved spacetimes by G\"obel in \cite{gobel} and by Hawking-King-McCarthy in \cite{Hawking-Topology}, and the discussion was taken into a further level by several authors, including an important feedback by Low in \cite{Low_path} (for a recent survey on the topologization of spacetime, see \cite{KBP_top}). In particular, Low proved that the Limit Curve Theorem (LCT) does not hold under the Path topology of Hawking-King-McCarthy, and claimed that this fact gives the right for the manifold topology to be called a fruitful one, against any other known candidates. Given that the LCT plays an important role in building contradictions in proofs in singularity theorems  in general relativity (see, for example, \cite{Galloway1} and \cite{Clarke}), the authors of \cite{KBP_sing} posed the question   whether the singularity problem is a purely topological one, depending largely on the topology that one chooses to equip the spacetime with. For example, if one considers the space of timelike paths, the notion of convergence stays unaffected, if choosing either the Path topology $\mathcal{P}$ or the manifold topology $T_M$, while convergence in the space of causal paths under $T_M$ does not imply convergence in the space of causal paths under $\mathcal{P}$ (see \cite{Low_path}). Consider the causal cone  of an event $x$  in a spacetime $M$, that is, consider the union of the time-cone $C^T(x)$ with the light-cone $C^L(x)$. Consider now an open ball $B_\epsilon^d (x)$  of radius $\epsilon$ w.r.t.\ the manifold topology $T_M$ (where the distance is defined via an induced Riemann metric $d$). Consider the intersection $A = (C^T(x) \cup C^L(x)) \cap B_\epsilon^d (x)$; the sets $A$ form a local basis for a topology $Z^{LT}$ on $M$, where   convergence under $Z^{LT}$ and under $T_M$ stays unaffected in the space of causal paths (see \cite{KBP_sing}). Obviously, this statement cannot hold if $Z^{LT}$ is considered in the space of timelike paths.

	It is clear that the path topology $\mathcal{P}$ in \cite{Hawking-Topology} has several advantages over the manifold topology on a spacetime, $\mathcal{P}$ incorporates the causal, differential and smooth conformal structure of spacetime and, most importantly, the group of homeomorphisms of $\mathcal{P}$ is the conformal group. According to G\"obel (see \cite{gobel}), the reason for considering the Euclidean topology as a ``natural'' topology for the Minkowski spacetime (and, more generally, the manifold topology for a spacetime manifold) was that people were mostly concerned with Riemannian spaces and not with spaces equipped with a Lorentzian metric; it really seems that the blind use of the manifold topology, while proving theorems in spacetime geometry, was due to ignorance (in the words of G\"obel)!
	
	
	In the setting of globally hyperbolic spacetimes, Low introduced a list of interesting topological results, including the following two  (see \cite{Low}).
	
	\begin{proposition} \label{L1}
		A strongly causal spacetime $M$ is globally hyperbolic  if and only if the space of smooth inextendible causal curves $\mathcal{C}$ is Hausdorff.
	\end{proposition}
	
	\begin{proposition} \label{L2}
		$M$ is globally hyperbolic  if and only if $\mathcal{C}$ is metrizable.
	\end{proposition}
	
	The topologization of the space of smooth inextendible causal curves $\mathcal{C}$ is important, because it affects the topology induced on $C$, the space of causal geodesics. $C$ need not be a manifold but, in particular, if the spacetime $M$ is strongly causal, then $C$ is a smooth manifold with boundary, with a smooth structure inherited from the homogeneous tangent bundle $UM$\footnote{The quotient of the tangent bundle minus the zero section, is called the homogeneous tangent bundle over $M$.} (see \cite{Low}). The canonical lift of a smooth causal curve from $M$ to $CM$ (where $CM$ is the bundle of causal directions)  and the lift from $C$ to $CM$ results in a foliation of $CM$. The space of leaves of this foliation is a topological space, with the quotient topology coming from $UM$. The topology of this space need not be Hausdorff even if $M$ is strongly causal. In the case that $C$, under this topology, is non-Hausdorff, then $M$ will be nakedly singular (a TIP lies inside a PIP; see \cite{Low} and for a more general exposition see \cite{Penrose-difftopology}).

	Low suggests a topology $T^0$  on $\mathcal{C}$  which can be described as follows. Given a smooth inextendible causal curve $\gamma$ in $\mathcal{C}$, let $\Gamma$ be the corresponding curve in $M$. Let $x \in \Gamma$ and  $U$ be an open neighborhood of $x$  in $M$, (w.r.t.\ the manifold topology $T_M$). Let $\mathcal{U} \subset \mathcal{C}$   be the set consisting of all (smooth inextendible  causal-) curves which pass through $U$. Let $T^0$   be the topology generated by $\mathcal{U}$. Then $T^0$ is    a  pointwise convergent topology. It follows,   that  if a spacetime $M$ has a closed timelike curve (CTC), then $(M,T^0)$ cannot satisfy the $T_1$-separation axiom and, if $M$ is totally vicious\footnote{For the definition and properties of (non-) totally vicious spacetimes, see \cite{Minguzzi1}.}, then there exists $x \in \mathcal{C}$ such that $x$ is dense in $\mathcal{C}$ w.r.t. $T^0$. For the proof of Proposition \ref{L2}, Low constructs a metric  which induces $T^0$  and remarked that one can find metrics which induce the topologies of convergence to any degree of smoothness (by means of the slicing of the jet bundles over the spacetime manifold $M$).
	
	Next, consider a spacetime $(M,g)$ of the form $M=\mathcal{I}\times\Sigma$, with $\mathcal{I}$ being an interval in $\R$ and $\Sigma$ a smooth $n$-dimensional manifold. The spatial slices $\Sigma_t=\Sigma\times\{t\}$ are spacelike submanifolds. We call such a spacetime a \emph{sliced space}. 
	In regards to the foliation, in \cite{ghc2}, sliced spaces were considered to have uniformly bounded lapse, shift and spatial metric, in order to achieve the equivalence of global hyperbolicity of $(M,g)$ with the completeness of the slice $(\Sigma, g_{\Sigma})$ (Theorem 2.1). Being motivated by this result, the authors of \cite{Kurt}  consider global topological conditions, for showing the equivalence of the global hyperbolicity of $(M,g)$ with the existence of a foliation 
	$(\Sigma_t,g_{\Sigma_t})$ being $T_A$-complete. Theorem \ref{Kurt}, below, differs from Theorem 2.1 in \cite{ghc2} in that the slices   are complete Riemannian manifolds (with uniformly bounded spatial metric, lapse  and shift  functions) while  the slices are $T_A$-complete in \cite[Theorem 2.1.]{Kurt}.

	
	\begin{theorem}\label{Kurt}
		Let $(M,g)$ be a sliced space, equipped with its natural product topology $T_P$, where $M = \mathbb{R}\times \Sigma$,  $\Sigma$ is an $n$-dimensional manifold ($n\ge 2$)and $g$ the $n+1$-Lorentz ``metric'' on $M$.  Let also $T_A$ be the Alexandrov spacetime topology on $M$. Then, the following statements are equivalent:\\
		\begin{enumerate}
			
			\item $(M,g)$ is globally hyperbolic.\\
			
			\item For every basic-open set $D \in T_A$, there exists a basic-open set $B \in T_P$, such that $D \subset B$.
			\\
			\item $(M_t,g_t)$ is complete with respect to $T_A$. 
			
		\end{enumerate}
		
	\end{theorem}
	$\,$\\	
	{{\bf{Question 2:}} As we have seen in this section, there are many different candidates for  spacetime topologies is, other than the manifold one. Is there a fruitful\footnote{Fruitful in the sense of \cite{Zeeman1}, \cite{gobel} and \cite{Hawking-Topology}. } and physically meaningful spacetime topology that was not considered so far? What are the criteria for choosing it? For example, how should such a topology be related to the foliation of a spacetime? These questions certainly need a more systematic work, because even if they can touch important problems (for example the singularity problem in relativity theory as well as Penrose's cosmic censorship conjecture), they seem underestimated until now.
		\\\\
		{{\bf{Question 3:}} In the light of Theorem \ref{Martinfinal}, we ask whether the topology of a spacetime can be reconstructed, in some technical and rigorous way, from the topology of an interval domain. The spacetime topologies that we have mentioned in this section, like for example the Zeeman fine topology or its generalization by G\"obel, or the one by Hawking-King-McCarthy are not metrizable; does this give us any insight on the appropriate candidate for a spacetime topology (other than the manifold topology), given that interval domains are order-theoretic objects?
			
			{
				\section{The Klein-Gordon Equation in Globally Hyperbolic Manifolds}
				In what follows, we  consider the Klein-Gordon operator on a Lorentzian manifold $(M,g)$ minimally coupled
				to an electromagnetic potential $A$ and with a scalar potential $Y$,
				\begin{align} K\phi:=
					\left(({\sqrt{|g|}})^{-1}D^A_{\mu}(\sqrt{|g|}g^{{\mu}{\nu}}D^A_{\nu})+Y\right)\phi=0\label{eq:kgop}
				\end{align}
				where ${|g|}=\text{det}(g_{\mu\nu})$ and $D^A_{\mu}=(i\partial-A)_{\mu}$.
				Next, we insert the form of the metric given in Theorem~\ref{T1} and, instead of working with the operator $K$ (from~\eqref{eq:kgop}), it is more convenient to work with the operator
				$$\tilde{K}:=N\,K\,N$$ with $N$ being as before the Lapse function and where   $\tilde{K}$ can be expressed as 
				\begin{equation}
					\tilde{K}=-(D_t+W^{*})(D_t+W)+L 
				\end{equation}
				with
				\begin{align*}
					W:= -A_0-\frac{1}{2}(N\,\vert g_{\Sigma} \vert^{1/2})^{-1}D_t\,(N\,\vert g_{\Sigma} \vert^{1/2}),
				\end{align*}
				and   $L(t)$ is the spatial Klein-Gordon operator given by 
				\begin{equation}\label{eqop}
					L(t)=-N\,(\sqrt{|g_{\Sigma}|})^{-1}(D-A)^{*}_{i}(   \sqrt{|g_{\Sigma}|}  \,N   g^{ij}_{\Sigma}\,    (D-A)_{j} )+  N^2 \,Y.
				\end{equation}
				Next, we define the operator $B$ by
				\begin{align*}
					B(t)=\begin{matrix}
						\left( \begin{array}{rrrr}
							W(t) & \I \,\,\,\,\,\, \\
							L(t) & W(t)   
						\end{array}\right).
					\end{matrix}
				\end{align*}
				Let  $u_1(t)=u(t)$ and $u_2(t)= -(D_t + W(t))\,u(t)$ be the Cauchy data for $u$ at time $t$, then 
				\begin{align*} 
					\left(\partial_t+iB(t)\right)
					\begin{matrix}
						\left(\begin{array}{r}
							u_1(t)  \\
							u_2(t)  
						\end{array}\right)=0
					\end{matrix}
				\end{align*}
				only if $u$ is a (weak)  solution of the Klein-Gordon equation $\tilde{K}u=0$. Then the Hamiltonian is given by the multiplication of $B$ with the matrix $Q$
				\begin{align*}
					Q=\begin{matrix}
						\left( \begin{array}{rrrr}
							0 & \I  \\
							\I  & 0   
						\end{array}\right)
					\end{matrix}
				\end{align*}
				i.e.~$H(t)=Q\,B(t)$. Since   $W(t)$ is simply a function the problem of proving self-adjointness of the Hamiltonian reduces to proving self-adjointness of  $L(t)$ for all   $t$. 
				
				In order to simplify  the nature of the problem, let us assume that the electromagnetic potential $A$ is equal to zero, the lapse function $N$ is equal to one and the metric $g_{\Sigma}$ is time independent. Let us further assume time independence and positivity for the scalar potential $Y$. By taking the assumptions into account, the operator $L$  simply is the standard Laplace-Beltrami operator w.r.t.\ the Riemannian manifold $(\Sigma,g_{\Sigma})$. The  results by  \cite{STR83} then state that    the geodesic completeness of the Riemannian manifold imply the essential self-adjointness of the Laplace-Beltrami operator  $L$. Since the Riemannian manifold stems from a globally hyperbolic spacetime it is a priori not clear why it should be geodesically complete. At this point let us state  \cite[Proposition~5.3]{K78}.
				\begin{proposition}\label{propk}
					Let  $(M,g)$ be a Lorentzian manifold  with metric tensor 
					\begin{equation}
						g=-dt^2+g_{\Sigma,ij}(\vec{x})dx^idx^j.
					\end{equation} 
					Then  the 	manifold $(M,g)$ is globally hyperbolic if and only if the Riemannian manifold  $(\Sigma,g_{\Sigma})$ is complete. 
				\end{proposition} 
				Next, it follows from Theorem \ref{T5} that the simplified Laplace-Beltrami operator we considered is an essentially self-adjoint operator on $C_0^ {\infty}(\Sigma)$. 
				
				In order to prove essential self-adjointness for the case when $N$ is not equal to one, in the absence of an electromagnetic potential,  we used techniques from weighted manifolds and reduced the problem to the following theorem  \cite[Theorem 4.1]{MO}.  
				\begin{theorem}\label{mt}  
					Let the Riemannian manifold $(\Sigma, N^{-2}g_{\Sigma})$   be geodesically  complete and let  the scaled potential     $N^2V \in L^2_{loc}(\Sigma)$ 	point-wise.  
					Furthermore, let the operator $L$ (from  (\ref{eqop})) be semi-bounded from below. Then  the operator $L$
					is   essentially self-adjoint  on $C_0^{\infty}(\Sigma)$.   
				\end{theorem}
				\begin{proof}
					For the proof see \cite{MO}.
				\end{proof}
				Next, we  generalize the former result to the case of nonvanishing electromagnetic potential and general lapse function. 
				\begin{theorem}
					Let the Riemannian manifold $(\Sigma,\tilde{g}_{\Sigma}:= N^{-2}g_{\Sigma})$   be geodesically  complete and let  the scaled potential     $N^2V \in L^2_{loc}(\Sigma)$ 	point-wise.  
					Furthermore,   let $A\in\Lambda^1_{(1)}(\Sigma)$ and let the  operator  \begin{equation*}
						L(t)=-N\,(\sqrt{|g_{\Sigma}|})^{-1}(D-A)^{*}_{i}(   \sqrt{|g_{\Sigma}|}  \,N   g^{ij}_{\Sigma}\,    (D-A)_{j} )+  N^2 \,Y.
					\end{equation*} be semi-bounded from below. Then  the operator $L$ (from   (\ref{eqop}))
					is   essentially self-adjoint  on $C_0^{\infty}(\Sigma )\subset L^{2}(\Sigma,\,\tilde{\mu})$. \end{theorem}
				\begin{proof}
					We   rewrite the operator $L$   as a weighted Laplace-Beltrami operator multiplied by the Lapse function, i.e. 
					\begin{align*}
						L=-N^2 {\Delta}^{A}_{{g}_{\Sigma},\mu}+N^2Y.
					\end{align*}
					By redefining   the metric and measure $\mu$ as  
					\begin{align*}
						\tilde{g}_{\Sigma}:= N^{-2}g_{\Sigma}, \qquad d\tilde{\mu}=N^{-2}d\mu ,
					\end{align*}
					the operator $L$ reads  (for the proof see \cite[Proposition 3.1]{MO})
					\begin{align*}
						L=-  {\Delta}^{A}_{\tilde{g}_{\Sigma},\tilde{\mu}}+N^2\,Y.
					\end{align*}
					By rewriting the operator as a minimally coupled, weighted Laplace-Beltrami operator, the condition of positivity of the potential and Theorem \ref{T5} imply essential self-adjointness. 
				\end{proof} 
				Hence, \emph{the problem of proving essential self-adjointness of the operator $L$    is reduced to proving geodesic completeness of the Riemannian manifold $(\Sigma, N^{-2}g_{\Sigma})$}. In
				all globally hyperbolic static spacetimes completeness holds due to Proposition \ref{propk}. In the case of
				stationary, globally hyperbolic spacetimes (i.e.\ with non-vanishing shift vector), Theorem \ref{mt} has been
				generalized under the assumption   that the Killing vector
				field $X(=\partial_t)  = N n(\Sigma) + N^i\partial_i$ is everywhere timelike (where $n(\Sigma)$ is a unit future-pointing normal	vector field on the hypersurface $\Sigma$),  \cite[Theorem 4.3.]{FMO}.
				\begin{theorem}\label{thm:51}
					Let $(M,g)$ be a stationary, globally hyperbolic spacetime with metric
					\begin{align*}
						g=-N^2(\vec{x})dt^2+g_{ij}(\vec{x})\,(dx^i+N^i\,dt)\,(dx^j+N^j\,dt).
					\end{align*}
					Then the Riemannian manifold w.r.t.\ the conformally transformed metric $(\Sigma, \mathbf{\tilde{h}})$ with
					\[ \tilde{h}_{ij}= N^{-2}{g}_{ij}+ (1-  N^{-2}\vec{N}^2)^{-1}N^{-4}N_iN_j \:, \]
					is complete.
				\end{theorem} In this context we refer the reader to  \cite{SR} as well, where global hyperbolicity of stationary metrics have been studied in the context of Finsler geometry. Furthermore, a full case study is given in \cite{SR2}, if the Killing vector field $X$ is not everywhere timelike.    
			\par
				Due to the geodesic completeness of the Riemannian manifold $(\Sigma,\mathbf{\tilde{h}})$, one can prove that the weighted Laplace-Beltrami operator that stems from the Klein-Gordon equation is essentially self-adjoint on the domain $C_0^{\infty}(\Sigma)$.
				
				In the case of non-stationary spacetimes, however, the connection between global hyperbolicity
				and self-adjointness of the spatial Laplacian is still unclear. Since any globally hyperbolic spacetime $(M,g)$ can be brought into the following form,
				\begin{equation*} 
					g=-N^2 (t,\vec{x})dt^2+g_{\Sigma,ij}(t,\vec{x})dx^idx^j,
				\end{equation*}
				and  keeping in mind that conformal transformations do not change the causal structure of the manifold, we may restrict attention to the conformally transformed spacetime $(M,\tilde{g})$ 
				(which is again globally hyperbolic) with
				\begin{align*} 
					\tilde{g}&=-  dt^2+N^{-2}g_{\Sigma,ij}(t,\vec{x})dx^idx^j\\&=-  dt^2+  \tilde{g}_{\Sigma,ij}(t,\vec{x})\,dx^idx^j
					.
				\end{align*}
				Our strategical reason  for the conformal transformation is  our interest in the connection of global hyperbolicity and the geodesic completeness of the Riemannian manifold $(\Sigma_t,\tilde{g}_{\Sigma_t})$ for all $t\in\R$. For three large classes of Lorentzian manifolds, the condition of global hyperbolicity is equivalent to  the geodesic completeness: $\,$\newline
				\begin{itemize}
					\item Static case. Lapse function and the spatial metric are time-independent and thus   Proposition~\ref{propk} leads to the equivalence. $\,$\newline 
					\item The case of warped manifolds $M=\R\times_f \Sigma$, see
					\cite[Theorem 3.66.]{BM96} or \cite[Lemma A.5.14.]{bgp}   where the following result  holds.\\ 
					\begin{theorem}
						Let $(\Sigma, g_{\Sigma})$ be a Riemannian manifold  and let $I = (a, b)$
						with $-\infty \leq a < b \leq +\infty$. Furthermore, let $f: I\mapsto (0,\infty)$ be a smooth function and the metric $g$ be given by the warped product $$g=-dt^2+f(t)\,g_{\Sigma}.$$
						Then,  the Lorentzian warped product $(I\times_f   \Sigma, g)$ 
						is globally hyperbolic iff $(\Sigma, g_{\Sigma})$
						is complete. 	
					\end{theorem} $\,$
					\item Sliced spaces. 
					Assume that the metric $g_{\Sigma,ij}(\vec{x},t)$ is uniformly bounded by the metric $g_{\Sigma,ij}(\vec{x},0)$ for all $t\in\mathbb{R}$ and tangent vectors $u\in T\Sigma$, i.e.\ that   there exist  constants $A,D\in\mathbb{R}>0$ such that			\begin{align*}\label{ineq:metAD}
						A\,g_{\Sigma,ij}(\vec{x},0)u^i\,u^j\leq g_{\Sigma,ij}(\vec{x},t)u^i\,u^j\leq D\,g_{\Sigma,ij}(\vec{x},0)u^i\,u^j.
					\end{align*}
					Then by  \cite[Theorem 2.1]{ghc2} or Theorem \ref{Kurt}, above,  the equivalence of globally hyperbolicity and geodesic completeness follows. 
				\end{itemize}$\,$\\
				In general however, i.e.\ apart from the three classes discussed above, there is no direct connection between geodesic completeness of the hypersurfaces and global hyperbolicity the Lorentzian manifold\footnote{We are indebted to Miguel S{\'a}nchez and Stefan Suhr for providing us with counter examples.}.  In particular,
				there are examples of non-globally hyperbolic spacetimes with geodesically complete hypersurfaces
				(for details see~\cite{MS97}). Nevertheless, in such spacetimes, the essential self-adjointness of the differential operator $L$ holds.
				This makes it possible to construct quantum field theories by employing methods in~\cite{NGH2}.
				In this way, one gets interesting new classes of quantum field theories in non-globally hyperbolic spacetimes.
				Analyzing the properties of the resulting quantum field theories seem interesting also in connection
				with black hole spacetimes and related phenomena.

			\end{document}